\documentclass[journal]{IEEEtran}
\usepackage{amssymb}
\usepackage[pdftex]{graphicx}
\usepackage{tikz}
\usepackage{framed}
\usepackage{bm}
\usepackage{cite}
\usepackage{comment}
\usepackage{algorithmic}
\usepackage{algorithm}
\usepackage{enumerate}
\usepackage{amsmath}
\usepackage{amsfonts}
\usepackage{color}
\usepackage{cases}

\newtheorem{theorem}{Theorem}

\newtheorem{remark}{Remark}
\newtheorem{lemma}{Lemma}
\newtheorem{corollary}{Corollary}

\newtheorem{proof}{Proof}

\usepackage{blkarray}

\newcommand{\argmin}{\mathop{\rm argmin}\limits}

\newcommand{\paren}[1]{\left(#1\right)}

\DeclareMathOperator{\rank}{rank}
\DeclareMathOperator{\diag}{diag}
\DeclareMathOperator{\tr}{tr}

\newcommand{\E}{\mathrm{e}}
\newcommand{\D}{\mathrm{d}}

\newcommand{\Real}{\mathbb{R}}

\newcommand{\bb}{\mathbb}

\usepackage{blkarray}

\usepackage{latexsym}
\def\qed{\hfill $\Box$} 

\usepackage{cite}
\usepackage{amsmath}
\usepackage{url}

\begin{document}
%
\title{Controllability scores for selecting control nodes of large-scale network systems}
\author{Kazuhiro Sato and Shun Terasaki\thanks{K. Sato and S. Terasaki are with the Department of Mathematical Informatics, Graduate School of Information Science and Technology, The University of Tokyo, Tokyo 113-8656, Japan, email: kazuhiro@mist.i.u-tokyo.ac.jp (K. Sato), t.tera1217@gmail.com (S. Terasaki) }}
\maketitle
\thispagestyle{empty}
\pagestyle{empty}

\begin{abstract}
To appropriately select control nodes of a large-scale network system,
we propose two control centralities called volumetric and average energy controllability scores.
The scores are the unique solutions to convex optimization problems formulated using the controllability Gramian.
The uniqueness is proven for stable cases and for unstable cases that include multi-agent systems.
We show that the scores can be efficiently calculated by using a proposed algorithm based on the projected gradient method onto the standard simplex.
Numerical experiments
demonstrate that
the proposed algorithm is more efficient than an existing interior point method, and
the proposed scores can correctly capture the importance of each state node on controllability, outperforming
 existing control centralities.
\end{abstract}


\IEEEpeerreviewmaketitle

\section{Introduction} \label{sec:intro}

To control large-scale network systems such as infection networks \cite{chaharborj2022controlling, zhu2021connectedness} and brain networks \cite{tang2018colloquium}, a selection of control nodes is a crucial step and seriously affects the system performance.
In this context, linear system $\dot{x}(t) = Ax(t)$
with state $x(t) =(x_1(t),\ldots, x_n(t))\in \Real^n$ and constant matrix $A\in \Real^{n\times n}$
can be regarded as a network system.
In fact, we can think that 
$A$ defines a network structure and
$x_1,\ldots, x_n$ represent the state nodes associated with the network.
Then, the selection problem can be interpreted as a problem to design a matrix $B\in {\bb R}^{n\times m}$ that satisfies
desired properties of
\begin{align}
    \label{eq:lti}
    \dot{x}(t) = Ax(t) + Bu(t).
\end{align}
In the selection problem, the following metrics using the Controllability Gramian (CG) have been frequently used to quantitatively evaluate the controllability of system \eqref{eq:lti}:
\begin{itemize}
    \item Trace of CG 
\cite{ikeda2018sparsity, romao2018distributed, sato2020controllability, she2021energy, summers2015submodularity, summers2016convex, fitch2016optimal, lindmark2021centrality}.
    \item Trace of Inverse CG \cite{she2021energy, summers2015submodularity, summers2016convex, fitch2016optimal, lindmark2021centrality}.
    \item Log Determinant of CG \cite{she2021energy, summers2015submodularity, summers2016convex, fitch2016optimal}.
    \item Smallest Eigenvalue of CG \cite{pasqualetti2014controllability, summers2015submodularity, summers2016convex, fitch2016optimal, bof2016role, lindmark2021centrality}.
\end{itemize}

Moreover, 
 the issue of selecting appropriate nodes in network systems is linked to the challenge of calculating node centralities. In the field of network science, centrality measures provide a way to quantify the relative importance of each node within the network, as detailed in \cite{rodrigues2019network}.
Specifically, by assessing these centrality values, we can identify the most influential or pivotal state nodes. These nodes are prime candidates for control nodes, due to their significant impact on the network systems.
However, as explained in \cite{summers2015submodularity} and further discussed  in Section \ref{Sec6-B},
many of the existing centrality measures fall short when applied to dynamical system \eqref{eq:lti}.
This discrepancy highlights the need for novel centrality measures.

To establish a more effective method for assessing the significance of each state node in
 dynamical network system \eqref{eq:lti},
 we introduce novel selection problems. 
These problems are grounded in continuous optimization approaches, characterized by objective functions such as the log determinant of CG and the trace of inverse CG.
This methodology aims to enhance the precision in identifying and prioritizing nodes based on their impact within the systems.
The main difference between our problem and other existing problems in \cite{ikeda2018sparsity,  pasqualetti2014controllability, romao2018distributed, sato2020controllability, she2021energy, summers2015submodularity, fitch2016optimal, bof2016role, lindmark2021centrality} 
is the assumption of the matrix $B$ in \eqref{eq:lti}.
Unlike existing works, we design $B$ by solving our optimization problems under the assumption that $B$ is diagonal with size $n$.
Although the diagonal assumption has already been introduced in \cite{olshevsky2014minimal, olshevsky2015minimum, olshevsky2017non}, 
we present the first interpretation of the assumption, emphasizing the importance of each state node.
The interpretation guarantees that we can use the solutions to our optimization problems as novel control centralities.
Notably, our problems essentially coincide with continuous relaxation problems considered in \cite{summers2016convex}
if the matrix $A$ in \eqref{eq:lti} is stable, i.e., the real parts of all the eigenvalues of $A$ are negative.
However, our problem formulations can be easily extended to unstable cases, which include multi-agent systems, as explained in Section \ref{Sec_unstable}.

The contributions of this paper are summarized as follows.
\begin{itemize}
    \item We propose novel concepts called Volumetric Controllabiity Score (VCS) and Average Energy Controllabiity Score (AECS).
    The scores quantitatively evaluate the importance of each state node on controllability.

    \item We show that under the assumption that the matrix $A$ in \eqref{eq:lti} is stable, \textcolor{black}{VCS and AECS uniquely exist in Theorems \ref{thm:unique_existence} and \ref{thm:strictlyconvex2}.}
    For unstable cases, we prove that the modifications of the VCS and AECS uniquely exist in Theorems \ref{thm:unstable_realeigen} and \ref{Thm_Laplacian}.
    
    \item We propose a simple optimization method with a theoretical guarantee, as shown in Theorem \ref{main_thm}, based on a projected gradient method.
    The computational complexity of the method is considerably smaller than the path following method \cite{vandenberghe1998determinant}, which is an interior point method, as explained in Section \ref{Sec5} \textcolor{black}{and demonstrated in Section \ref{Sec6-C}}.
    That is, our proposed method can be applied to large-scale systems unlike the path following method.
    
    \item We compare the \textcolor{black}{VCS, AECS,} and existing control centralities proposed in \cite{summers2015submodularity, liu2012control} and explain why the \textcolor{black}{VCS and AECS} are better than
    other control centralities in Section \ref{Sec6-B}.
\end{itemize}

The remainder of this paper is organized as follows.
In Section \ref{sec:prelim}, we introduce the concepts of CG and controllability ellipsoid,
and explain the relation.
In Sections \ref{sec:formulation} and \ref{Sec_unstable}, we formulate convex optimization problems and provide the definitions of VCS and AECS as the solutions.
The uniqueness of the solutions is proved for stable cases in Section \ref{sec:formulation} and for unstable cases in Section \ref{Sec_unstable}.
In Section \ref{Sec5}, we propose an algorithm with a convergence guarantee for solving the problems based on the projected gradient method.
In Section \ref{Sec6}, 
we demonstrate the effectiveness of our proposed algorithm compared with an interior point method.
Moreover,
we compare the proposed controllability scores with existing control centralities.
Finally, our conclusions are presented in Section \ref{sec:conclusion}.


{\it Notation:}
 The set of real numbers is denoted by ${\bb R}$.
 For a matrix $A\in {\bb R}^{m\times n}$, we define $\|A\|_{\rm F}$ as the Frobenius norm
$\|A\|_{\rm F}:=\sqrt{{\rm tr}(A^{\top}A)}$,
where $A^{\top}$ denotes the transpose of $A$, and
${\rm tr}(X)$ denotes the diagonal sum for a squared matrix $X$. 
Given a vector $v = (v_i)\in {\bb R}^n$, $\|v\|$ and ${\rm diag} (v_1,\ldots, v_n)$ denote the usual Euclidean norm $\|v\|=\sqrt{v^{\top}v}$ and the diagonal matrix with the diagonal elements $v_1,\ldots, v_n$, respectively.
Instead of ${\rm diag} (v_1,\ldots, v_n)$, we also use ${\rm diag} (v)$.
For a symmetric matrix $A$,
the symbols $A\succeq O$ and $A\succ O$ denote the positive semi-definie and definite matrices, respectively.
The symbol $L^2[0,T]$ is the set of all measurable functions $u\colon [0,T]\to \Real^m$ that satisfy $\int_0^T \|u(t)\|^2 \D t < \infty$.

\section{Preliminaries}
\label{sec:prelim}

In this section, we summarize the concepts of CG and controllability ellipsoid to formulate our problems.

The finite-time CG of system \eqref{eq:lti} is defined as 
\begin{align*}
    \mathcal{W}(T) := \int_0^T \exp(At) BB^\top \exp(A^\top t) \ \D t \succeq O, 
\end{align*}
where $T>0$.
System \eqref{eq:lti} is controllable if and only if $\mathcal{W}(T)\succ O$ for all $T > 0$.
If the matrix $A$ of system \eqref{eq:lti} is stable, $\mathcal{W}(T)$ has a limit with $T\to \infty$ and the limit $\mathcal{W} := \lim_{T\to\infty} \mathcal{W}(T)$ satisfies Lyapunov equation
\begin{align}
    A\mathcal{W} + \mathcal{W} A^\top = -BB^\top. \label{eq:lyapunov}
\end{align}
Then, system \eqref{eq:lti} is controllable if and only if {$\mathcal{W} \succ O$}.

If system \eqref{eq:lti} is controllable,
the reachable space of system \eqref{eq:lti} defined as 
\begin{align*}
    \mathcal{E}(T) = \left\{ x(T) \in \Real^n \left|
    \begin{array}{l}
        u\in L^2[0, T],\,
       \int_0^T \|u(t)\|^2 \D t\leq 1,   \\
        \dot{x}(t) = Ax(t) + Bu(t),\,
        x(0) = 0
    \end{array}
    \right. \right\}
\end{align*}
 can be expressed by
\begin{align}
    \label{eq:controllability_ellipsoid}
    \mathcal{E}(T) = \{y\in\Real^n\mid y^\top \mathcal{W}(T)^{-1}y \leq 1\},
\end{align} 
as shown in \cite[Section 4.3]{dullerud2013course}. 
The volume of $\mathcal{E}(T)$ is proportional to $\sqrt{\det \mathcal{W}(T)}$, as shown in \cite[Corollary 12.15]{vishnoi2021}.
Moreover, if the matrix $A$ in \eqref{eq:lti} is stable,
we can consider $\mathcal{E}(T)$ with $T\to\infty$, and \eqref{eq:controllability_ellipsoid} can be replaced with 
\begin{align*}
\mathcal{E} = \{y\in\Real^n \mid y^\top \mathcal{W}^{-1} y\leq 1\}.
\end{align*}
We call $\mathcal{E}(T)$ and $\mathcal{E}$ controllability ellipsoids.


\section{Controllability scoring problem and controllability score}
\label{sec:formulation}

In this section, we introduce the concepts of VCS and AECS, which are novel control centralities.
Throughout this section, we assume the following:
\begin{enumerate}[(i)]
    \item The matrix $A$ of system \eqref{eq:lti} is known.
\item The matrix $B$ is of the form 
\begin{align}
B={\rm diag}(\sqrt{p_1},\ldots, \sqrt{p_n}). \label{B_form}
\end{align}
    \item The matrix $A$ is stable.
\end{enumerate}

The assumption (i) means that the matrix $A$ must be pre-identified.
The assumption (ii) means that input node $u_i$ directly affects state node $x_i$ if $p_i>0$. 
The assumption (iii) is made to simplify discussions and is removed in Section \ref{Sec_unstable}.

Under the assumptions (i), (ii), and (iii), by introducing a nonnegative vector $p := (p_1,\dots, p_n)$ and using the relation $BB^\top = \diag(p)$, the solution to \eqref{eq:lyapunov} can be expressed as
\begin{align}
    W_{\rm con}(p) &:= \sum_{i=1}^n p_i W_i, \label{eq:pgramian}
\end{align}
where $W_i$ is the CG defined as
\begin{align}
    \label{eq:horizon_gramian}
    W_i := \int_0^\infty \exp(At) e_ie_i^\top \exp(A^\top t) \ \D t.
\end{align}
Here, $e_i\in\Real^n$
 denotes the standard vector, which has $1$ at $i$-th position and zeros at other positions. 

\subsection{Volumetric Controllability Score (VCS)} \label{Sec3A}
We consider a problem that maximizes $\log\det W_{\rm con}(p)$, because $\sqrt{\det W_{\rm con}(p)}$ is proportional to the volume of the controllability ellipsoid $\mathcal{E}$, as explained in Section \ref{sec:prelim}.
Because input node $u_i$ of system \eqref{eq:lti} with \eqref{B_form}
 has a one-to-one correspondence with state node $x_i$, we can expect that $p$ contains information on the importance of each state node on controllability. 
However, if there is no constraint on $p$, the volume can become infinitely large.
That is, unless constraints are placed on $p$, it is difficult to relate $p$ to the importance of each state node on controllability.

Thus, we consider the following optimization problem.
\begin{framed}
\vspace{-1em}
\begin{align}
    \label{prob:org}
    \begin{aligned}
        &&& \text{minimize} && f(p) := -\log\det W_{\rm con}(p) \\
        &&& \text{subject to} && p \in \mathcal{F} := X\cap \Delta.
    \end{aligned}
\end{align}
\vspace{-1em}
\end{framed}

\noindent
Here,
\begin{align}
    X & := \{p\in\Real^n \mid W_{\rm con}(p)\succ O\}, \nonumber 
    \\
    \Delta & := \left\{ p = (p_i) \in \Real^n \left|
    \begin{array}{l}
         \sum_{i=1}^n p_i = 1,    \\
         0 \leq p_i \quad (i=1,\dots,n)  
    \end{array}
    \right. \right\}, \label{eq:simplex}
\end{align}
where $X$ is an open set in $\Real^n$, because the map $p\in \Real^n \mapsto W_{\rm con}(p)$ is continuous.
Moreover, $X$ and $\Delta$ are convex sets in $\Real^n$.
The constraint $p\in X$ is needed for defining the objective function $f(p)$.
    In other words, if $p\not\in X$, we cannot use the objective function.
    Here, $p\in X$ means that system \eqref{eq:lti} with \eqref{B_form} is controllable.

Moreover, if the optimal solution $p^*$ to problem \eqref{prob:org} uniquely exists,
each element of $p^*$ can be interpreted as
the relative importance assigned to each state node from the viewpoint of controllability. This is because the assignment is under the conditions that
$p^*$ is restricted on $\Delta$ and the volume of the controllability ellipsoid $\mathcal{E}$ is maximized.
Thus, by referring to $p^*$, we can select control nodes.

In the following, we show that the optimal solution to problem \eqref{prob:org} uniquely exists (Theorem \ref{thm:unique_existence}).
However, this is not trivial,
because 
the feasible region $\mathcal{F}$ is neither closed nor open.
In fact, consider
    $A = \begin{pmatrix}
        -1 & 0 \\
        1 & -1
\end{pmatrix}$.
In this case, $A$ is stable and CG \eqref{eq:pgramian} is given by
    $W_{\rm con}(p) = p_1\begin{pmatrix}
        \frac{1}{2} & \frac{1}{4} \\
        \frac{1}{4} & \frac{1}{4}
    \end{pmatrix} +
    p_2\begin{pmatrix}
        0 & 0 \\
        0 & \frac{1}{2}
    \end{pmatrix}$.
Therefore, the feasible region $\mathcal{F}$ can be expressed as 
    $\mathcal{F} = \left\{ \begin{pmatrix}
    p_1 \\
    p_2
    \end{pmatrix}
     \in\Real^2 \ \middle| \  p_1 > 0, \ p_2 \geq 0,\ p_1+p_2 = 1 \right\}$, which is neither closed nor open in $\Real^2$.

However, for any $p^{(0)}\in \mathcal{F}$,
we can define
\begin{align}
    \mathcal{F}_0 := \{p\in \Real^n\mid f(p) \leq f(p^{(0)})\} \cap \mathcal{F}, \label{def:F0}
\end{align}
and have the following lemma.

\begin{lemma} \label{Lem_F0}
For any $p^{(0)}\in \mathcal{F}$,
\begin{align}
    \mathcal{F}_0=\{p\in \Real^n\mid f(p) \leq f(p^{(0)})\} \cap \Delta, \label{eq_F0}
\end{align}
and $\mathcal{F}_0$ is a closed, bounded, and convex set in ${\bb R}^n$.
\end{lemma}
\begin{proof}
See Appendix \ref{Ape_B}. \qed
\end{proof}

Problem \eqref{prob:org} is a minimization problem, and has a feasible solution $p^{(0)}\in \mathcal{F}$.
In fact, $\left(1/n,\dots, 1/n\right)\in \mathcal{F}$.
Thus, problem \eqref{prob:org} is equivalent to
\begin{framed}
\vspace{-1em}
\begin{align}
    \label{prob:org2}
    \begin{aligned}
        &&& \text{minimize} && f(p) \\
        &&& \text{subject to} && p \in \mathcal{F}_0.
    \end{aligned}
\end{align}
\vspace{-1em}
\end{framed}

Moreover,
the objective function $f(p)$ of problems \eqref{prob:org} and \eqref{prob:org2} is twice-differentiable on the open set $X$, and for $p\in X$,
the gradient $\nabla f(p)$ and Hessian $\nabla^2 f(p)$ are given by
\begin{align}
    \left(\nabla f(p)\right)_{i} & = -\tr \left(W_{\rm con}(p)^{-1} W_i \right), \label{eq:deriv} \\
    \left(\nabla^2 f(p)\right)_{ij} & = \tr \left(W_{\rm con}(p)^{-1} W_i W_{\rm con}(p)^{-1} W_j \right), \label{eq:sndderiv}
\end{align}
respectively.
Using \eqref{eq:sndderiv}, we have the following theorem.
\begin{theorem}
    \label{thm:strictlyconvex}
    The objective function $f(p)$ of problems \eqref{prob:org} and \eqref{prob:org2} is strictly convex on $\mathcal{F}$.
\end{theorem}
\begin{proof}
See Appendix \ref{Ape_A}. \qed
\end{proof}


From the above argument, we obtain the following theorem.

\begin{theorem} \label{thm:unique_existence}
The solution to problem \eqref{prob:org} uniquely exists.
\end{theorem}
\begin{proof}
See Appendix \ref{Ape_B}. \qed
\end{proof}

Theorem \ref{thm:unique_existence} guarantees that there exists the unique optimal solution $p^\ast$ to problem \eqref{prob:org}. 
We call $p_i^\ast$ Volumetric Controllability Score (VCS) of state node $x_i$ and $p^\ast$ VCS of $x$.
Moreover, we call problem \eqref{prob:org} the controllability scoring problem.

A larger VCS of state node $x_i$
indicates greater importance of the node from a controllability perspective.
This can be understood from the fact that for any $p\in X$,
\begin{align*}
\left(\nabla f(p)\right)_{i} = -{\rm tr} (W_{\rm con}(p)^{-1/2} W_i W_{\rm con}(p)^{-1/2}) <0,
\end{align*}
because $W_i\neq 0$.
That is, $\left(\nabla f(p)\right)_{i}<0$ on $X$ guarantees  that an increase in $p_i$ leads to a decrease in the value of $f$, thereby increasing the volume of the controllability ellipsoid $\mathcal{E}$.

To compute the VCS, we propose an algorithm for solving problem \eqref{prob:org}
in Section \ref{Sec5}.

\begin{remark}
Problem \eqref{prob:org} explicitly involves the constraint on $p$.
Moreover, $\det W_{\rm con}(p)$ in the problem can be related to the controllability ellipsoid $\mathcal{E}$, as explained in Section \ref{sec:prelim}.
Note that the definition of $\mathcal{E}$ includes the input energy constraint $\int_0^\infty \|u(t)\|^2 \D t\leq 1$.
That is, when we consider problem \eqref{prob:org},
the input energy constraint is implicitly assumed.
\end{remark}

\begin{remark} \label{Rem_VCE}
    The Volumetric Control Energy (VCE) centrality proposed in \cite{summers2015submodularity} is defined as
\begin{align}
    C_{\mathrm{VCE}}(i) := \log \prod_{j=1}^{k_i} \lambda_j\left(W_i\right)=\sum_{j=1}^{k_i}\log \lambda_j(W_i), \label{eq:vce}
\end{align}
where $i\in \{1,\ldots,n\}$, $k_i:=\rank W_i$, and $\lambda_1(W_i),\ldots, \lambda_{k_i}(W_i)$ denote the eigenvalues satisfying $\lambda_1(W_i)\geq \cdots \geq \lambda_{k_i}(W_i)>0$ of $W_i$.
Here, {
$k_i$ means the controllable subspace dimension of system \eqref{eq:lti} with $B=e_i$.}
{Note that if $k_i=n$, $C_{\mathrm{VCE}}(i) = \log \det W_i = -f(e_i)$.}
That is, VCS and VCE are completely different.

Due to the definition \eqref{eq:vce}, even if $k_i>k_j$, $C_{\rm VCE}(i)<C_{\rm VCE}(j)$ may hold.
This means that even if $k_i$ is large, $\lambda_{k_i}(W_i)$ may be almost zero.
Thus, VCE does not always capture the importance of each node from the viewpoint of controllability.
More concretely, see Section \ref{Sec6-B}.
\end{remark}

\subsection{Average Energy Controllability Score (AECS)}

Instead of controllability scoring problem \eqref{prob:org}, we can consider the following problem.
\begin{framed}
\vspace{-1em}
\begin{align}
    \label{prob:trinv}
    \begin{aligned}
        &&& \text{minimize} && g(p) := \tr \paren{W_{\rm con}(p)^{-1}} \\
        &&& \text{subject to} && p \in \mathcal{F}.
    \end{aligned}
\end{align}
\vspace{-1em}
\end{framed}

The objective function $g(p)$ in \eqref{prob:trinv} can be interpreted as the average energy required for the system to be driven, as explained in \cite{olshevsky2017non}.
In fact, $g(p)$ is proportional to the average of $x_f^\top W_{\rm con}(p)^{-1} x_f$ with respect to $x_f$ over a uniform distribution of the unit sphere.
Here, the function $x_f^\top W_{\rm con}(p)^{-1} x_f$ means the minimum squared input energy $\int_0^\infty \|u(t)\|^2 \D t$ subject to $x(0)=0$ and $\lim_{t\rightarrow \infty} x(t) =x_f$, as shown in \cite[Section 4.3]{dullerud2013course}.

The function $g(p)$ is also twice-differentiable on the open set $X$, and for $p\in X$,
the gradient $\nabla g(p)$ and Hessian $\nabla^2 g(p)$ are given by
\begin{align}
  &  \left( \nabla g(p)\right)_i = -{\rm tr}\left(W_{\rm con}(p)^{-1}W_iW_{\rm con}(p)^{-1} \right), \label{gradient_g}\\
   & \left( \nabla^2 g(p) \right)_{ij} = {\rm tr} \left(W_{\rm con}(p)^{-1}W_jW_{\rm con}(p)^{-1}W_i W_{\rm con}(p)^{-1} \right) \nonumber \\
     &\quad \quad +{\rm tr} \left(W_{\rm con}(p)^{-1}W_i W_{\rm con}(p)^{-1}W_j W_{\rm con}(p)^{-1} \right), \label{hessian_g}
\end{align}
respectively.
Using \eqref{hessian_g}, we have the following theorem.

\begin{theorem} \label{thm:strictlyconvex2}
    The objective function $g(p)$ of problem \eqref{prob:trinv} is strictly convex on $\mathcal{F}$.
\end{theorem}
\begin{proof}
See Appendix \ref{Ape_A}. \qed
\end{proof}

Similarly to Theorem \ref{thm:unique_existence}, by using Theorem \ref{thm:strictlyconvex2}, we can prove that there exists the unique optimal solution $p^\ast$ to problem \eqref{prob:trinv}. 
Thus, each element of the solution $p^\ast$ can also be interpreted as the relative importance assigned to each state node from the viewpoint of controllability.
We call
$p_i^*$ Average Energy Controllability Score (AECS) of state node $x_i$ and
$p^\ast$ AECS of $x$.
Note that, similarly to VCS, a larger AECS for state node $x_i$
indicates greater importance of the node from a controllability perspective, because $\left( \nabla g(p)\right)_i<0$ for any $p\in X$.
Additionally, we refer to problem \eqref{prob:trinv} as the controllability scoring problem.

The difference between VCS and ACES can be summarized as follows:
\begin{itemize}
    \item VCS of each state node indicates its importance in enlarging the controllability ellipsoid $\mathcal{E}$.

    \item AECS of each state node indicates its importance in steering the overall state to a point on the unit sphere.
\end{itemize}

\begin{remark} \label{Rem_ACE}
    The Average Control Energy (ACE) centrality proposed in \cite{summers2015submodularity} is defined as    
\begin{align*}
C_{\mathrm{ACE}}(i) := -\tr\left(W_i^\dagger\right),
\end{align*}
where $i\in \{1,\ldots,n\}$ and $W_i^\dagger$ is the pseudo inverse of $W_i$.
Similarly to the case of VCS and VCE, AECS and ACE are completely different,
and
 ACE does not always capture the importance of each node from the viewpoint of controllability.
More concretely, see Section \ref{Sec6-B}.
\end{remark}

\section{Finite-time controllability scoring problem} \label{Sec_unstable}

In this section,
 we only assume (i) and (ii) presented in Section \ref{sec:formulation}.
That is, the matrix $A$ of system \eqref{eq:lti} may be unstable.
Even in this case, we can consider controllability scoring problems using the finite-time CG
\begin{align}
    W_{\rm con}(p, T)  
     = \sum_{i=1}^n p_i W_i(T) \label{Def_Wc}
\end{align}
with
\begin{align}
    \label{eq:finite_time_horizon_gramian}
    W_i(T) := \int_0^T \exp(At) e_ie_i^\top \exp(A^\top t) \ \D t,
\end{align}
and convex objective functions
\begin{align}
f_T(p) &:= -\log\det W_{\rm con}(p, T), \label{def_fT}\\
g_T(p) &:= \tr \paren{W_{\rm con}(p, T)^{-1}}. \label{def_gT}
 \end{align}
That is, by replacing $X$ in \eqref{prob:org} and \eqref{prob:trinv} with 
\begin{align*}
X_T:= \{p\in {\bb R}^n \mid W_{\rm con}(p,T)\succ O\},
\end{align*}
we can consider the following convex optimization problem called the finite-time controllability scoring problem.
\begin{framed}
\vspace{-1em}
 \begin{align}
    \label{prob:unstable}
    \begin{aligned}
        &&& \text{minimize} && f_T(p)\,\,\, {\rm or}\,\,\,  g_T(p) \\
        &&& \text{subject to} && p \in X_T\cap \Delta.
    \end{aligned}
\end{align}
\vspace{-1em}
\end{framed}

Unlike the stable case in Section \ref{sec:formulation}, the uniqueness of the optimal solution to \eqref{prob:unstable} is not guaranteed in general. In fact, consider the unstable matrix
$A = \begin{pmatrix}
        0 & \omega \\
        -\omega & 0
\end{pmatrix}$, $\omega\neq 0$.
Then, we have
    $W_1(T)  = \begin{pmatrix}
        \frac{1}{2}\left( T +\frac{\sin2\omega T}{2\omega} \right) & \frac{1}{4}(1-\cos 2\omega T)  \\
        \frac{1}{4}(1-\cos 2\omega T)& \frac{1}{2}\left(         T -\frac{\sin2\omega T}{2\omega}\right)
    \end{pmatrix}$ and 
    $W_2(T)  = \begin{pmatrix}
       \frac{1}{2}\left( T -\frac{\sin2\omega T}{2\omega} \right) & -\frac{1}{4}(1-\cos 2\omega T)  \\
      -\frac{1}{4}(1-\cos 2\omega T) & \frac{1}{2}\left( T +\frac{\sin2\omega T}{2\omega} \right)
    \end{pmatrix}$.
Thus, if $T=T^\ast := \pi /\omega$,
    $W_1(T^\ast) = W_2(T^\ast) = \frac{1}{2\omega}\begin{pmatrix}
    \pi & 0 \\
    0 & \pi
    \end{pmatrix}
    \succ O$.
Hence, for all $p \in \mathcal{F}$, we have
    $f_{T^\ast}(p) = -\log\det (p_1+p_2)W_1(T^\ast)  = -\log\det W_1(T^\ast)$.
That is, if $T = T^\ast$, $f_{T^\ast}(p)$ is constant for all $p\in X_T\cap \Delta$ and the optimal solution to controllability scoring problem \eqref{prob:unstable} is not unique.

Fortunately, finite-time controllability scoring problem \eqref{prob:unstable} for some important unstable cases has an unique solution.

\begin{theorem}
    \label{thm:unstable_realeigen}
     Assume that $A$ and $-A$ of system \eqref{eq:lti} do not have a common eigenvalue.
    Then, there exists the unique solution to
    finite-time controllability scoring problem \eqref{prob:unstable}.
\end{theorem}
\begin{proof}
See Appendix \ref{Ape_C}.\qed
\end{proof}

The following theorem means that we can use the unique solution to finite-time controllability scoring problem \eqref{prob:unstable}
to select leaders in a multi-agent system.
The selection has been studied in \cite{abbas2020tradeoff, fitch2016optimal, she2021energy}.
\textcolor{black}{Note that a graph Laplacian matrix $L$ has $0$ eigenvalue. That is, because $L$ and $-L$ have the common eigenvalue, we cannot use Theorem \ref{thm:unstable_realeigen}.}

\begin{theorem}[Laplacian dynamics] \label{Thm_Laplacian}
    Let $L$ be a graph Laplacian matrix corresponding to an undirected connected graph. Then, for $A=-L$,
    there exists a unique solution to
 finite-time controllability scoring problem \eqref{prob:unstable}.
\end{theorem}
\begin{proof}
See Appendix \ref{Ape_C}. \qed
\end{proof}


\section{Algorithm for Controllability Scores} \label{Sec5}
Algorithm \ref{alg:projgrad} is a projected gradient method for calculating the VCS, which is the unique solution to controllability scoring problem \eqref{prob:org}.
\textcolor{black}{Step 4 in Algorithm~\ref{alg:projgrad} is a terminal condition, which is called an $\varepsilon$-stationally point condition.
The step size $\alpha^{(k)}$ is defined by using Algorithm \ref{alg:Armijo}, which is called Armijo rule along the projection arc \cite{bertsekas2016nonlinear}.
Algorithm \ref{alg:Armijo} terminates after a finite number of iterations.
}

As noted in Section~\ref{Sec3A}, the feasible region $\mathcal{F}$ of problem \eqref{prob:org} is not a closed set of $\Real^n$ in general, and thus we cannot define the projection onto $\mathcal{F}$. 
Hence, instead of the projection onto $\mathcal{F}$, Algorithm~\ref{alg:projgrad} uses the projection onto the standard simplex $\Delta$ in \eqref{eq:simplex}, which includes $\mathcal{F}$,
\begin{align*}
    \Pi_{\Delta}(q) := \argmin_{p\in\Delta}\|p-q\|^2.
\end{align*}
Note that the projection $\Pi_{\Delta}$ is efficiently computed\cite{condat2016fast}.

Thus, Algorithm \ref{alg:projgrad} is not a standard projected gradient method for solving \eqref{prob:org}, and the following is non-trivial.

\begin{theorem} \label{main_thm}
    \textcolor{black}{Let $\{p^{(k)}\}$ be a sequence generated by Algorithm \ref{alg:projgrad} with $\varepsilon=0$.
    Then, 
    \begin{align}
        \lim_{k\rightarrow \infty} p^{(k)} = p^*, \label{convergence_score}
    \end{align}
    where $p^*$ is the optimal solution to controllability scoring problem \eqref{prob:org}, that is, the VCS.}
\end{theorem}
\begin{proof}
See Appendix \ref{Ape_E}. \qed
\end{proof}

As shown in the proof of Theorem \ref{main_thm}, we obtain the following corollary, which 
means that Algorithm \ref{alg:projgrad} can be regarded as a projected gradient method for solving \eqref{prob:org2}, that is, controllability scoring problem \eqref{prob:org}.

\begin{corollary}
 The projection $\Pi_{\Delta}$ onto the standard simplex $\Delta$ in Algorithms \ref{alg:projgrad} and \ref{alg:Armijo} is
 the projection $\Pi_{\mathcal{F}_0}$ onto the set $\mathcal{F}_0$ defined as \eqref{def:F0}.
\end{corollary}

To use Algorithm~\ref{alg:projgrad},
we need 
CGs $W_1,\ldots, W_n$ in \eqref{eq:horizon_gramian}, which can be obtained by solving \eqref{eq:lyapunov} based on the Bartels--Stewart method \cite{bartels1972solution}.
However, this method has a time complexity of $O(n^3)$, making it impractical for
 $n>1000$.
In such cases, an approximation method for solving \eqref{eq:lyapunov} is advisable. The CF-ADI algorithm \cite{li2002low} is an efficient alternative,
particularly because each controllability Gramian $W_i$ is defined through a low-rank matrix $B=e_i$.
However, the accuracy of this approximation depends on the system matrix $A$.

\begin{figure}[!t]
\begin{algorithm}[H]
    \caption{A projected gradient method}
    \label{alg:projgrad}
  \textbf{Input:} Controllability Gramians $W_1,\ldots, W_n$ in \eqref{eq:horizon_gramian}, $p^{(0)} := (1/n,\ldots, 1/n) \in\mathcal{F}$, and $\varepsilon\geq 0$.\\
 \textbf{Output:} {Approximate VCS.}
    \begin{algorithmic}[1]
    \FOR{$k=0,1,\ldots$}
    \STATE $q^{(k)} := p^{(k)}- \alpha^{(k)} \nabla f(p^{(k)})$, where $\alpha^{(k)}$ is defined by using Algorithm \ref{alg:Armijo}. 
    \STATE $p^{(k+1)} := \Pi_{\Delta}(q^{(k)})$.

        \IF{$\|p^{(k)}-p^{(k+1)}\|\leq \varepsilon$}
        \RETURN $p^{(k+1)}$.
        \ENDIF
    \ENDFOR
    \end{algorithmic}
\end{algorithm}
\end{figure}

\begin{figure}[!t]
\begin{algorithm}[H]
    \caption{Armijo rule along the projection arc}
    \label{alg:Armijo}
    \textbf{Input:} $\sigma,\,\rho\in (0,1)$ and $\alpha>0$.\\
    \textbf{Output:} {Step size.}
    \begin{algorithmic}[1]
    \STATE $\tilde{p}^{(k)}:=\Pi_{\Delta}(p^{(k)}-\alpha \nabla f(p^{(k)}))$.
    \IF{$f(\tilde{p}^{(k)})\leq f(p^{(k)}) + \sigma \nabla f(p^{(k)})^\top (\tilde{p}^{(k)}-p^{(k)})$}
        \RETURN $\alpha^{(k)}:=\alpha$.
    \ELSE
    \STATE $\alpha\leftarrow \rho \alpha$, and go back to step 1.
        \ENDIF
    \end{algorithmic}
\end{algorithm}
\end{figure}


\begin{remark} \label{remark_interior}
Problem \eqref{prob:org} is a special case of the max-det problem studied in \cite{vandenberghe1998determinant}, which can be solved using the path following method.
However, the time complexity at each iteration of the path following method is larger than $O(n^4)$,
because the method uses the Hessian given by \eqref{eq:sndderiv} in addition to the gradient \eqref{eq:deriv}.
In contrast, Algorithm \ref{alg:projgrad} only uses the gradient.
Thus, the time complexity at each iteration of Algorithm \ref{alg:projgrad} is $O(n^3)$.
In Section \ref{Sec6-C},
we compare the execution times of Algorithm~\ref{alg:projgrad}
and the path following method through numerical experiments.
Moreover, the complexity of VCE defined by \eqref{eq:vce} is $O(n^4)$, because it requires calculating the eigenvalues of each $W_i$. That is, the computation of VCE is not more efficient than that of VCS.
\end{remark}

\begin{remark}
For solving controllability scoring problem \eqref{prob:trinv},
$f$ in Algorithms \ref{alg:projgrad} and \ref{alg:Armijo} should be replaced with $g$,
where the gradient $\nabla g(p^{(k)})$ is given by \eqref{gradient_g}.
Similarly to Theorem \ref{main_thm}, we can prove that the modified algorithm outputs the global optimal solution to problem \eqref{prob:trinv}.
Moreover, the time complexity at each iteration of the modified algorithm is the same as that of  Algorithm \ref{alg:projgrad}.
\end{remark}

\begin{remark}
For unstable systems considered in Section \ref{Sec_unstable},
Algorithms \ref{alg:projgrad} and \ref{alg:Armijo} can be modified as follows.
\begin{enumerate}
    \item We replace $W_i$ with $W_i(T)$ in \eqref{eq:finite_time_horizon_gramian}.
    \item We replace in $f$ with $f_T$ in \eqref{def_fT} or with $g_T$ in \eqref{def_gT}.
\end{enumerate}
\end{remark}

\section{Numerical experiments} \label{Sec6}
In this section,
we examine the time performance of Algorithm \ref{alg:projgrad} compared with the path following algorithm proposed in \cite{vandenberghe1998determinant}
using network systems up to $n=1000$.
Moreover, we
compare our proposed VCS and AECS with the existing control centralities VCE (See Remark \ref{Rem_VCE}), ACE (See Remark \ref{Rem_ACE}), and the Control Capacity (CC) proposed in \cite{liu2012control} using a small-size network system with $n=10$.
We here note that
VCS and AECS are not directly comparable with ACE, VCE, and CC. That is, the comparison of numerical quantities lacks significance unless their physical interpretations are taken into account.
We explain in this section that VCS and AECS are shown to be superior than ACE, VCE, and CC in
their abilities to determine the importance of each state node on
controllability under their physical interpretations. 
Throughout all numerical experiments, we used $\varepsilon =10^{-4}$ in Algorithm \ref{alg:projgrad}.

\subsection{Time performance of Algorithm \ref{alg:projgrad}} \label{Sec6-C}
Table \ref{table2} shows the execution time of Algorithm \ref{alg:projgrad} and the path following algorithm in \cite{vandenberghe1998determinant} for the system size $n=200$, $400$, and $1000$.
 The blank column for the path following algorithm means that the calculation was not finished in
one day.
As illustrated in the table, Algorithm 1 was considerably better than the path following method.
This result is consistent with the analysis in Remark \ref{remark_interior}.

\begin{table}[t]
\caption{Computational time (in seconds) of Algorithm 1 and the path following method \cite{vandenberghe1998determinant}. } \label{table2}
  \begin{center}
    \begin{tabular}{|c|c|c|c|} \hline
                 $n$ & $200$  & $400$ & $1000$  \\ \hline 
      Algorithm 1    & $1.112$  & $14.34$  & $72.98$  \\ \hline 
      Path following method \cite{vandenberghe1998determinant}   &  $225.2$  & $5449$  &  \\\hline 
    \end{tabular}
  \end{center}
\end{table}

\subsection{Comparison to other control centralities} \label{Sec6-B}
We considered a stable network with $n=10$, as depicted in Fig.~\ref{fig:erdos_n10}. 
The edge weights are determined by uniform distribution on $[0,1]$.
Here, we omitted the visualizations of self-loops  to make it easier to view the network structure.
That is, we cannot use the simple formula for the CC proposed in \cite{liu2012control}
based on the layer index of each node, because the formula can be applied when the network has no self-loops.

As shown in Table \ref{table:centrality}, the VCS tends to provide the importance of upstream nodes more highly than the AECS.
In other words,
the AECS can give a non-trivial result compared with the VCS, because it is obvious that
upstream nodes in a network with a hierarchical structure are important.
That is, the AECS can be used to determine the value of the importance of each state node in the network, which cannot be decided by the network structure alone.

According to Table \ref{table:centrality}, state node 7 was the most important for the VCS, AECS, and CC.
This seems correct, because node 7 is connected to the most state nodes, as shown in Fig. \ref{fig:erdos_n10}. 
In contrast, node 7 was not important for
the VCE and ACE.
Instead, node 10 was the most important for the VCE, and nodes 5, 6, and 8 were the most important for the ACE.
The results seem strange according to Fig. \ref{fig:erdos_n10}, because nodes 5, 6, and 8 do not influence other nodes and node 10 only affects node 6, which is also influenced by node 4.
However, we can understand that the strange results on the VCE and ACE are due to the definitions, {as mentioned in Remarks \ref{Rem_VCE} and \ref{Rem_ACE}.}


\begin{figure}[t]
    \centering
    \includegraphics[width=4.5cm]{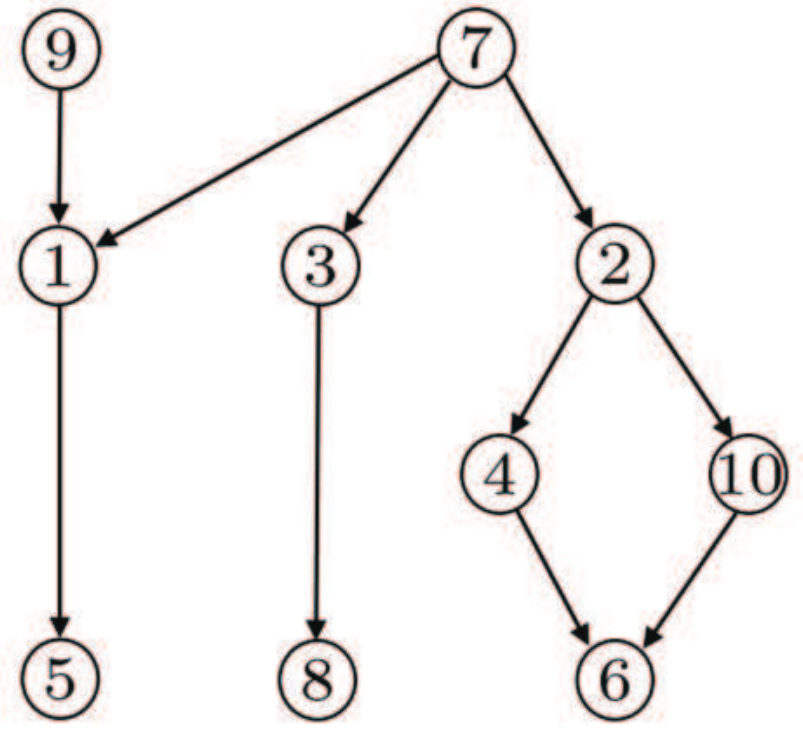}
    \caption{Network structure with $n=10$.}
    \label{fig:erdos_n10}
\end{figure}

\begin{table}[t]
\caption{Control centralities for the system.}
\label{table:centrality}
\centering
\begin{tabular}{c||c|c|c|c|c}
Node $i$ & VCS & AECS & CC & ACE & VCE \\ \hline
1 & 0.000 & 0.138 & 2 & -1.936 & 0.818 \\
2 & 0.151 & 0.202 & 3 & -3.501 & -69.050 \\
3 & 0.187 & 0.200 & 2 & -1.881 & 0.895 \\
4 & 0.005 & 0.051 & 2 & -4.792 & -0.768 \\
5 & 0.000 & 0.000 & 1 & -0.600 & 0.511 \\
6 & 0.000 & 0.000 & 1 & -0.600 & 0.511 \\
7 & {\bf 0.361} & {\bf 0.237} & 9 & -5.524 & -100.395 \\
8 & 0.000 & 0.000 & 1 & -0.600 & 0.511 \\
9 & 0.296 & 0.172 & 3 & -5.315 & -70.905 \\
10 & 0.000 & 0.000 & 2 & -1.795 & 1.031
\end{tabular}
\end{table}

Moreover, it should be noted that although the CC seems to be able to detect the most important node on controllability,
the value of the CC for each node may not be important.
In fact, as shown in Table \ref{table:centrality}, the second largest value of the CC was 3 of nodes 2 and 9, although the largest value was 9 of node 7.
That is, the difference between the first and second values of the CC was considerably large unlike the proposed VCS and AECS.
This suggests that although the controllable subspace dimension of node 9 is larger than those of other nodes,
a large control input energy to control the state is required.
Furthermore, although node 3 was not important for the CC, the node was relatively important for the VCS and AECS.

More concretely,
although the controllable subspace dimension corresponding to node 7 was 9 according to the CC,
the minimum positive eigenvalue of $W_7$ was very small.
Thus, node 7 for the VCE and ACE was not important.
In contrast to the VCE and ACE, the VCS and AECS use information of $W_1,\ldots, W_{10}$ not only of $W_7$ to decide the importance of node 7.
Therefore, VCS and AECS enable us to provide correct scores unlike the ACE, VCE, and CC.

\section{Concluding remarks} \label{sec:conclusion}

We have proposed two novel control centralities called the volumetric and average energy controllability scores, abbreviated by VCS and AECS, respectively.
We have shown that these scores are the unique solutions to convex optimization problems using the controllability Gramians for both stable and unstable cases.
Through numerical experiments, we have demonstrated that the proposed controllability scores and calculating algorithm are better than existing ones.

Using the duality of controllability and observability \cite{kalman1963mathematical},
we can define observability scores, which express the importance of each state node for observing the whole state, by considering 
        $\dot{x}(t) = A x(t)$ and $y(t) =C x(t)$
with a diagonal matrix $C\in {\bb R}^{n\times n}$ as well as the controllability scores using system \eqref{eq:lti}.
That is, our approach provides a novel method for determining sensor placements based on the concept of observability \cite{yan2015spectrum, bopardikar2021randomized, liu2013observability, manohar2022optimal}.

Although we can use 
the VCS and AECS proposed in this paper to consider candidates of control nodes,
it is also important to apply
other metrics.
In fact, for example, as shown in Table \ref{table:centrality} in Section \ref{Sec6},
the CC provided the consistent results with the VCS and AECS.
Because the time complexity for calculating the CC is smaller than those of the VCS and AECS, we can use the CC to know a rough tendency of the importance of each node before checking the VCS and AECS.

The extension of Theorem \ref{Thm_Laplacian} to a directed graph case is non-trivial, because the proof relies on the orthogonal decomposition of a symmetric matrix corresponding to an undirected graph.
Thus, the extension presents an interesting mathematical problem and is left as future work.

Finally, we note that VCS and AECS proposed in this paper can be defined under the assumption that the system matrix $A$ is known.
In other words, if $A$ is unknown, VCS and AECS cannot be defined.
In a future work, we will develop a data-driven method, which is related to \cite{fotiadis2022data}, for calculating VCS and AECS under the assumption that $A$ is unknown.


\section*{Acknowledgment}
This work was supported by the Japan Society for the Promotion of Science KAKENHI under Grant 20K14760 and 23K03899.

\appendix

\subsection{Proof of Theorems \ref{thm:strictlyconvex} and \ref{thm:strictlyconvex2}} \label{Ape_A}

To prove Theorems \ref{thm:strictlyconvex} and \ref{thm:strictlyconvex2}, we show the following.

\begin{lemma} \label{Lem_Wi_independent}
Let $W_i$ $(i=1,\ldots, n)$ be defined as \eqref{eq:horizon_gramian}.
Then, $W_1,\ldots, W_n$ are linearly independent over ${\bb R}$.
\end{lemma}
\begin{proof}
The linear independence of $W_1,\ldots, W_n$ is equivalent to that
 $W_{\rm con}(x)=O$ yields $x=0$,
because $W_{\rm con}(\cdot)$ is defined as \eqref{eq:pgramian}.
From \eqref{eq:lyapunov}, each controllability Gramian $W_i$ satisfies $A W_i + W_i A^\top = -e_i e_i^\top$,
and thus
\begin{align}
    A W_{\rm con}(x) + W_{\rm con}(x) A^\top = -\diag(x_1,\dots, x_n).
\end{align}
Therefore, $W_{\rm con}(x)=O$ implies $x=0$.\qed
\end{proof}

{\it Proof of Theorem \ref{thm:strictlyconvex}:}
Suppose $p\in \mathcal{F}\subseteq X$. From \eqref{eq:sndderiv}, for all $x \in \Real^n\setminus\{0\}$, 
\begin{align}
    x^\top (\nabla^2 f(p))x 
     = \tr(G(p, x)^2), \label{eq:quadtrace}
\end{align}
where 
\begin{align}
   G(p, x) := W_{\rm con}(p)^{-\frac{1}{2}}W_{\rm con}(x)W_{\rm con}(p)^{-\frac{1}{2}}. \label{eq:defofG}
\end{align}
Since both $W_{\rm con}(p)$ and $W_{\rm con}(x)$ are symmetric matrices, $G(p,x)$ is a symmetric matrix and $G(p, x)^2$ is a positive-semidefinite matrix. Thus, $x^\top (\nabla^2 f(p)) x\geq 0$ and $f(p)$ is convex on $\mathcal{F}$.

To prove the strict convexity of $f(p)$ on $\mathcal{F}$, we suppose that $x^\top (\nabla^2 f(p)) x = 0$.
Then, \eqref{eq:quadtrace} yields $\tr\left( G(p, x)^2\right) = 0$, and thus $G(p, x) = 0$. 
From the nonsingularity of $W_{\rm con}(p)$ on $\mathcal{F}$ and \eqref{eq:defofG}, $W_{\rm con}(x) = 0$, which implies
$x=0$ from Lemma \ref{Lem_Wi_independent}.
Thus, $f(p)$ is strictly convex on $\mathcal{F}$.
\qed

{\it Proof of Theorem \ref{thm:strictlyconvex2}:}
Suppose $p\in \mathcal{F}\subseteq X$. From \eqref{hessian_g}, for all $x \in \Real^n\setminus\{0\}$, 
\begin{align}
     x^\top (\nabla^2 g(p))x  
     = 2{\rm tr} \left( W_{\rm con}(p)^{-1/2}G(p,x)^2 W_{\rm con}^{-1/2}(p) \right) \geq 0, \label{quad_hessian_g}
\end{align}
where $G(p,x)$ is defined as \eqref{eq:defofG}.
Thus, $g(p)$ is convex on $\mathcal{F}$.

To prove the strict convexity of $g(p)$ on $\mathcal{F}$,
we suppose that $x^{\top} (\nabla^2 g(p)) x=0$.
Then,
\eqref{quad_hessian_g} implies
\begin{align}
    W_{\rm con}(p)^{-1} G(p,x)^2=0, \label{quad_hessian_g2}
\end{align}
because $x^{\top}(\nabla^2 g(p))x=2{\rm tr}(W_{\rm con}(p)^{-1} G(p,x)^2)$, $W_{\rm con}(p)^{-1}\succ O$, and $G(p,x)^2\succeq O$.
Thus, it follows from \eqref{quad_hessian_g2} that
\begin{align}
    W_{\rm con}(x)W_{\rm con}(p)^{-1}W_{\rm con}(x) = 0. \label{kaname}
\end{align}
Because $W_{\rm con}(p)\succ O$, \eqref{kaname} implies $W_{\rm con}(x)=0$,
which yields $x=0$ from Lemma \ref{Lem_Wi_independent}. 
Thus, $g(p)$ is strictly convex on $\mathcal{F}$.
\qed

\subsection{Proof of Lemma \ref{Lem_F0} and Theorem \ref{thm:unique_existence}} \label{Ape_B}

{\it Proof of Lemma \ref{Lem_F0}:}
First, we show that \eqref{eq_F0} holds.
Let $Y:=\{p\in \Real^n\mid f(p) \leq f(p^{(0)})\}$ and $p\in Y \cap \Delta$.
Then, 
\begin{align}
f(p)\leq f(p^{(0)}). \label{f_sub_level}
\end{align}
If $p\not\in X$, the value of $f$ is not defined.
This is a contradiction to \eqref{f_sub_level}.
Thus, $p\in X$, which means that $\mathcal{F}_0 \supset Y\cap \Delta$.
The converse clearly holds from the definition of $\mathcal{F}$. 

The boundedness of $\mathcal{F}_0$ follows from that of $\Delta$, and
the convexity follows from that of $\Delta$ and $Y$.
So, we show that $\mathcal{F}_0$
is a closed set in ${\bb R}^n$.
The set $\Delta$ is a closed set in ${\bb R}^n$.
Thus, if $Y$ is closed, $\mathcal{F}_0$ is also closed.
To show that $Y$ is closed,
we take any sequence $\{p_k\}$ converges to $p\in {\bb R}^n$.
By the definition of $Y$, $\{p_k\}\subset X$.
The continuity of the objective function $f$ on $X$ implies
$f(p)= f(\lim_{k\rightarrow \infty} p_k) = \lim_{k\rightarrow \infty} f(p_k)\leq f(p^{(0)})$.
This means that $p\in Y$,
and thus $Y$ is a closed set in ${\bb R}^n$.
This completes the proof.
\qed

{\it Proof of Theorem \ref{thm:unique_existence}:}
From Lemma \ref{Lem_F0}, $\mathcal{F}_0$ is a closed and bounded set in ${\bb R}^n$.
Moreover, the function $f$ is continuous on $\mathcal{F}_0$.
Therefore, problem \eqref{prob:org2} has an optimal solution, as shown in \cite[Proposition A.8]{bertsekas2016nonlinear}.
Furthermore, Theorem \ref{thm:strictlyconvex} and Lemma \ref{Lem_F0} guarantee that $f$ is strictly convex and $\mathcal{F}_0$ is a convex set in ${\bb R}^n$.
This means that an optimal solution to problem \eqref{prob:org2} is unique, as shown in \cite[Proposition 1.1.2]{bertsekas2016nonlinear}. \qed

\subsection{Proof of Theorems \ref{thm:unstable_realeigen} and \ref{Thm_Laplacian}} \label{Ape_C}

To prove Theorem \ref{thm:unstable_realeigen},
we prepare the following Lemma and its corollary.

\begin{lemma} \label{Lem_Sylvester}
Let $A\in {\bb R}^{m\times m}$ and $B\in {\bb R}^{n\times n}$.
The Sylvester equation
$AX-XB = C$
    has a unique solution for any $C\in {\bb R}^{m\times n}$ if and only if $A$ and $B$ do not have a common eigenvalue.
\end{lemma}
\begin{proof}
This is a well-known result. See \cite[Theorem 2.4.4.1]{horn2012matrix}. \qed
\end{proof}

Lemma \ref{Lem_Sylvester} immediately implies the following corollary.

\begin{corollary} \label{Cor_Sylvester}
    Assume that  $A$ and $-A$ of system \eqref{eq:lti} do not have a common eigenvalue.
    Then, 
    \begin{align}
        X=O\quad \Leftrightarrow\quad AX+XA^{\top} = O. \label{X=O_equiv}
    \end{align}
\end{corollary}

\textcolor{black}{Corollary \ref{Cor_Sylvester} means that the assumption of Theorem \ref{thm:unstable_realeigen} is a sufficient condition for \eqref{X=O_equiv} to be satisfied.}

{\it Proof of Theorem \ref{thm:unstable_realeigen}:}
     It is sufficient to show that the controllability Gramian $W_1(T),\ldots, W_n(T)$ are linearly independent over ${\bb R}$,
     because this implies the strict convexities of $f_T(p)$ and $g_T(p)$ in the same way
     as the proofs of Theorems \ref{thm:strictlyconvex} and \ref{thm:strictlyconvex2}, respectively.
     As a result, in the same way as Theorem \ref{thm:unique_existence}, the linear independence guarantees that there exists the unique solution to
    finite-time controllability scoring problem \eqref{prob:unstable}.

     The linear independence of $W_1(T),\ldots, W_n(T)$ is equivalent to that $W_{\rm con}(x,T)=O$ yields $x=0$, since $W_{\rm con}(p,T)$ is defined as \eqref{Def_Wc}. The fundamental theorem of calculus, that is,
     $F(T) - F(0) = \int_0^T \frac{ \D F}{\D t}(\tau) \D \tau$
     with $F(t):=\E^{AT}e_ie_i^\top \E^{A^\top T}$,
     implies that $W_i(T)$ satisfies the Lyapunov equation
    \begin{align}
        \label{eq:finite_horizon_lyapunov}
        AW_i(T) + W_i(T) A^\top = \E^{AT} e_i e_i^\top \E^{A^\top T} - e_i e_i^\top.
    \end{align}
    Thus, multiplying $x_i$ to both sides of \eqref{eq:finite_horizon_lyapunov} and adding them for $i = 1,\dots, n$, we get
    \begin{align}
        \label{eq:finite_lyapunov}
         AW_{\rm con}(x, T) + W_{\rm con}(x, T) A^\top =
         \E^{AT} \diag(x) \E^{A^\top T} - \diag(x).
    \end{align}   
  \textcolor{black}{Corollary \ref{Cor_Sylvester} yields}
$W_{\rm con}(x, T) = O$ if and only if $AW_{\rm con}(x, T) + W_{\rm con}(x, T) A^\top = O$.
    This and \eqref{eq:finite_lyapunov} imply that $W_{\rm con}(x,T)=O$ if and only if
    \begin{align}
        \label{eq:finite_lyapunov_right}
        \E^{AT} \diag(x)  - \diag(x) \E^{-A^\top T}= O.
    \end{align}
 Because
 $A$ and $-A$ of system \eqref{eq:lti} do not have a common eigenvalue,
 $\E^{AT}$ and $\E^{-A^\top T}$ do not also share any eigenvalue.
 Thus, Corollary \ref{Cor_Sylvester} implies that
 \eqref{eq:finite_lyapunov_right} holds if and only if $x=0$.
 Therefore, $W_{\rm con}(x,T)=0$ yields $x=0$. \qed

 {\it Proof of Theorem \ref{Thm_Laplacian}:}
    As in the proof of Theorem \ref{thm:unstable_realeigen}, it is sufficient to show that
    \begin{align}
        \label{eq:finite_lyapunov_right2}
        \E^{AT} \diag(x)  - \diag(x) \E^{-A^\top T}= O
    \end{align}
    has the unique solution $x = 0$. Eq.~\eqref{eq:finite_lyapunov_right2} can be rewritten as
    \begin{align}
        p_{ij} x_j - q_{ij} x_i &= 0 \quad (i\neq j), \label{eq:nondiag_eq} \\
        (p_{ii} - q_{ii}) x_i &= 0 \quad (i= j), \label{eq:diag_eq}
    \end{align}
    where $p_{ij} := (\E^{AT})_{ij}$ and $q_{ij} := (\E^{-A^\top T})_{ij}$.
    
    We first show that $p_{ij} = (e^{AT})_{ij} > 0$ for all $i, j\in\{1,\dots, n\}$. From the definition of $A = -L$, there exists $\alpha > 0$ such that $P:= A+\alpha I$ is a nonnegative matrix.
    Thus, $\E^{AT} = \E^{PT} \E^{-\alpha T}$, and it follows from \cite[Theorem 1.1.2]{bapat1997nonnegative} that the connectivity of $G$ implies that $\E^{PT}$ is a positive matrix. This means $p_{ij} > 0$ for all $i, j\in\{1,\dots, n\}$, and \eqref{eq:nondiag_eq} implies that 
    \begin{align}
        x_j = (q_{ij}/p_{ij}) x_i \quad (i\ne j) \label{eq:relation_xi_xj}
    \end{align}
    
    The orthogonal decomposition of the symmetric matrix $L$ yields $L = U^\top DU$,
    where $U=(u_{ij})$ is an orthogonal matrix, and $D = \diag(\lambda_1,\dots, \lambda_n)$ is a diagonal matrix with eigenvalues of $L$\ ($\lambda_1 = 0< \lambda_2 \leq \dots \leq \lambda_n$).
    Then, we have
        $p_{ii}  = (\E^{AT})_{ii} = \sum_{k=1}^n u_{ik}^2 \E^{-\lambda_k T}$,
        $q_{ii} = (\E^{-A^\top T})_{ii} = \sum_{k=1}^n u_{ik}^2 \E^{\lambda_k T}$,
    and
    $q_{ii} - p_{ii}  = \sum_{k=1}^n u_{ik}^2 \left( \E^{\lambda_k T} - \E^{-\lambda_k T} \right) \geq 0$.
    Suppose that $q_{ii} - p_{ii} = 0$ for any $i\in\{1,\dots, n\}$. Then, we have $u_{ik} = 0\ (k\ne 1)$, which contradicts the nonsingularity of $(u_{ij})$. 
    Therefore, there is $i\in\{1,\dots, n\}$ such that $q_{ii} \ne p_{ii}$ and \eqref{eq:diag_eq} yields $x_i = 0$.
    Combining $x_i = 0$ and \eqref{eq:relation_xi_xj}, we have $x_1=\dots=x_n=0$. \qed

\subsection{Proof of Theorem \ref{main_thm}} \label{Ape_E}

    We define $\alpha^{(k)}$ in step 2 of Algorithm \ref{alg:projgrad} using Algorithm \ref{alg:Armijo}.
    As shown in \cite[Proposition 3.3.3]{bertsekas2016nonlinear},
    there exists $\alpha^{(k)}>0$ such that
    \begin{align}
        f({p}^{(k+1)})\leq f(p^{(k)}) + \sigma \nabla f(p^{(k)})^\top ({p}^{(k+1)}-p^{(k)}). \label{f_evaluation}
    \end{align}
    Note that $\nabla f(p^{(k)})^\top ({p}^{(k+1)}-p^{(k)})\leq 0$, because
    $(p^{(k)}-p^{(k+1)})^\top (p^{(k)}-\alpha^{(k)}\nabla f(p^{(k)}) - p^{(k+1)})\leq 0$, which follows from the projection theorem described in \cite[Proposition 1.1.4]{bertsekas2016nonlinear}.
    Thus, \eqref{f_evaluation} yields
        $f(p^{k+1}) \leq f(p^{(k)})$,
    which means that 
    \begin{align}
    p^{(k)}\in \mathcal{F}_0 \subset \mathcal{F}\quad (k=0,1,\ldots), \label{p_F0}
    \end{align}
     where $\mathcal{F}_0$ is defined as \eqref{def:F0}.

Lemma \ref{Lem_F0} and \eqref{p_F0} imply
$\Pi_{\mathcal{F}_0}(p^{(k)}- \alpha^{(k)} \nabla f(p^{(k)})) = \Pi_{\Delta}(p^{(k)}- \alpha^{(k)} \nabla f(p^{(k)}))$,    
where $\Pi_{\mathcal{F}_0}$ denotes the projection onto $\mathcal{F}_0$.
That is, Algorithm~\ref{alg:projgrad} is a projected gradient method for solving
 convex optimization problem \eqref{prob:org2}.
    Therefore, \cite[Theorem 2]{iusem2003convergence} guarantees that \eqref{convergence_score} holds.\qed





\bibliographystyle{IEEEtran}
\end{document}